\documentclass[12pt,oneside]{amsart}
\usepackage{amscd,amsmath,amssymb,amsfonts,a4}
\usepackage{pstricks}
\usepackage{color}
\newlength{\rulebreite}

\def\timesover#1#2#3{\ \xymatrix@1@=0pt@M=0pt{ _{#1}&\times&_{#2} \\& ^{#3}&}\ }
\def\otimesover#1#2#3{\ \xymatrix@1@=0pt@M=0pt{ _{#1}&\otimes&_{#2} \\& ^{#3}&}\ }
\usepackage[all]{xy}
\theoremstyle{plain}
\newtheorem{thm}{Theorem}
\newtheorem{lem}[thm]{Lemma}
\newtheorem{cor}[thm]{Corollary}
\newtheorem{prop}[thm]{Proposition}
\newtheorem{prop-ex}[thm]{Example}
\theoremstyle{definition}
\newtheorem{defn}[thm]{Definition}

\newtheorem{rmk}[thm]{Remark}

\newtheorem{ex}[thm]{Example}

\newtheorem{eig}[thm]{Properties}

\newtheorem{notas}[thm]{Notations}
\newtheorem{construction}[thm]{Construction}

\numberwithin{thm}{section}
\numberwithin{equation}{section}

\newcommand{\ml}[2]{\begin{multline}\label{#1}#2 \end{multline}}
\newcommand{\ga}[2]{\begin{gather}\label{#1}#2 \end{gather}}
\newcommand{\Spec}{{\rm Spec \,}}

\newcommand{\ilim}{\mathop{\varprojlim}\limits} 
\newcommand{\surj}{\twoheadrightarrow}
\newcommand{\inj}{\hookrightarrow}
\newcommand{\tensor}{\otimes}

\newcommand{\sC}{{\mathcal C}}
\newcommand{\sD}{{\mathcal D}}

\newcommand{\sO}{{\mathcal O}}

\newcommand{\sT}{{\mathcal T}}

\newcommand{\A}{{\mathbb A}}

\newcommand{\C}{{\mathbb C}}

\newcommand{\N}{{\mathbb N}}

\newcommand{\Z}{{\mathbb Z}}
\newcommand{\et}{{\acute{e}t}}

\begin{document}
\title[Algebraic fundamental group]{On the algebraic fundamental group of smooth varieties in characteristic $p>0$}
\author{H\'el\`ene Esnault}
\address{
Universit\"at Duisburg-Essen, Mathematik, 45117 Essen, Germany}
\email{esnault@uni-due.de}
\author{Amit Hogadi}
\address{ School of Mathematics, Tata Institute of Fundamental Research, Homi Bhabha Road, Colaba, Mumbai 400005, India}
\email{amit@math.tifr.res.in}
\date{April 9, 2010}
\thanks{Partially supported by  the DFG Leibniz Preis, the SFB/TR45, the ERC Advanced Grant 226257}
\begin{abstract}We define an analog in characteristic $p>0$ of the proalgebraic completion of the topological fundamental group of a complex manifold. 
\end{abstract}
\maketitle
\section{Introduction} 
Let $X$ be a smooth algebraic variety defined over a field $k$ endowed with a rational point $x\in X(k)$.

If $k$ is the field of complex numbers $\C$,
the proalgebraic completion 
$\pi^{{\rm alg,rs}}(X,x)$ of the topological fundamental group $\pi^{{\rm top}}_1(X,x)$ is defined as  the prosystem $\varprojlim H$, where $H\subset GL(n, \C)$ runs over the Zariski closures of the monodromy groups $\rho(\pi^{{\rm top}}(X,x))$ of complex linear representations $\rho: \pi^{{\rm top}}_1(X,x) \to GL(n,\C)$. The profinite completion $\varprojlim H$, where $H$ runs over the finite quotients of $\pi^{{\rm top}}_1(X,x)$, is, via the Riemann existence theorem, identified with Grothendieck's \'etale fundamental group $\pi_1^{{\rm \et}}(X,x)$. Since any finite group is embeddable in $GL(n, \C)$ for some $n$, this defines, thinking of $\pi_1^{{\rm \et}}(X,x)$ as a complex (constant) proalgebraic group, a surjective homomorphism 
$\varphi_{\C}^{{\rm rs}}: \pi^{{\rm alg,rs}}(X,x)\to \pi_1^{{\rm \et}}(X,x)$, and in fact $\pi_1^{{\rm \et}}(X,x)$ is the profinite quotient of $\pi^{{\rm alg,rs}}(X,x)$. 
By the Riemann-Hilbert correspondence, $\pi^{{\rm alg,rs}}(X,x)$ is the Tannaka group-scheme of the category of $\sO_X$-coherent regular singular $\sD_X$-modules, which is a full subcategory of the category of $\sO_X$-coherent $\sD_X$-modules. We denote by $\pi^{{\rm alg}}(X,x)$ the corresponding Tannaka group-scheme, and by $\varphi_{\C}: \pi^{{\rm alg}}(X,x) \surj \pi^{{\rm alg,rs}}(X,x)\xrightarrow{\varphi_{\C}^{{\rm rs}}} \pi_1^{{\rm \et}}(X,x)$ the composite morphism. It is surjective as well, and since any flat connection with finite monodromy is regular singular,  $\pi_1^{{\rm \et}}(X,x)$ is the profinite quotient of $\pi^{{\rm alg}}(X,x)$. 

If $k$ is a characteristic $0$ field,  $\pi^{{\rm alg}}(X,x)$ is defined as the Tannaka group-scheme of the $k$-linear tensor category of $\sO_X$-coherent $\sD_X$-modules equipped with the fiber functor defined as the restriction of the module on $x$. The full subcategory of {\it finite objects}, that is   objects  with finite monodromy group-scheme, or said differently, objects which have the property that the full Tannaka subcategory which is spanned by it has a finite Tannaka group-scheme, defines a pro-finite $k$-group-scheme $\pi^{{\rm \et}}(X,x)$. Since $\pi^{{\rm \et}}(X,x)(\bar k)=\pi_1^{\rm \et}(X,x)$ (\cite[Remark~2.10]{EHS}), and both $\pi^{{\rm alg}}(X,x)$ and $\pi^{{\rm \et}}(X,x)$ satisfy base change for finite extensions $k\subset L$ (\cite[Property~2.54)]{EH}), we see that 
the surjection $\varphi: \pi^{{\rm alg}}(X,x)\to \pi^{{\rm \et}}(X,x)$ is a $k$-form of $\varphi_{\C}$ for any complex embedding $k\subset \C$. 
Moreover, by definition, $\varphi$ induces the pro-finite quotient homomorphism.  

If $k$ is a  characteristic $p>0$ field, the category of $\sO_X$-coherent $\sD_X$-modules is again a $k$-linear abelian tensor rigid category. It is part of Katz' theorem asserting that this category is equivalent to the 
category of stratified $\sO_X$-coherent sheaves (see \cite[Theorem~1.3]{G}, \cite[Theorem~8]{dS}, where it is shown over $k=\bar k$). If $k=\bar k$, 
its Tannaka group-scheme $\pi^{{\rm alg}}(X,x)$ is shown to be pro-smooth in \cite[Corollary~12]{dS} (strictly speaking, it is shown there  only for the profinite part, but dos Santos' proof applies more generally as mentioned in \cite[Corollary~7]{dSOW}). The homomorphism $\varphi$ is then defined by the full embedding of the subcategory of objects with finite monodromy group-scheme. So by definition, $\varphi$ induces the pro-finite quotient homomorphism.  

 On the other hand, if $X$ is a reduced connected scheme over a characteristic $p>0$ field $k$, endowed with a rational point $x\in X(k)$, Nori \cite[Chapter~II]{N}
constructed a fundamental group-scheme $\pi^N(X,x)$ as the projective system of finite $k$-group-schemes $G$ for which there is a  $G$-torsor $h:Y\to X$ under $G$ with trivialization at $x$. The pro-\'etale quotient of $\pi^N(X,x)$ is precisely $\pi^{{\rm \et}}(X,x)$.  

Summarizing, one has a diagram 
\begin{equation}\label{diag1.1}
\xymatrix{\pi^{{\rm alg}}(X,x) \ar[r]^{{\rm surj}} & \pi^{{\rm \et}}(X,x) \\
& \pi^N(X,x) \ar[u]_{{\rm surj}}
}
\end{equation}
The aim of our article is to define a Tannaka category ${\sf Strat}(X, \infty)$ over a perfect field $k$, which contains the category of $\sO_X$-coherent $\sD_X$-modules as a full subcategory, in such a way that its Tannaka group-scheme $\pi^{{\rm alg}, \infty}(X,x)$, which thus surjects onto $\pi^{{\rm alg}}(X,x)$, also surjects onto $\pi^N(X,x)$. In other words, we complete \eqref{diag1.1} to 
\begin{equation}\label{diag1.2}
\xymatrix{\pi^{{\rm alg}}(X,x) \ar[r]^{{\rm surj}} & \pi^{{\rm \et}}(X,x) \\
\ar[u]^{{\rm surj}} \pi^{{\rm alg},\infty}(X,x) \ar[r]^{{\rm surj}} & \pi^N(X,x) \ar[u]_{{\rm surj}}
}
\end{equation}
As a byproduct, we obtain a purely tannakian geometric description of $\pi^N(X,x)$ (see Corollary \ref{cor4.7}). Recall that we assume that $X$ is smooth. If in addition $X$ is proper, Nori himself described his fundamental group-scheme $\pi^N(X,x)$ as the Tannaka group-scheme of the category of essentially finite bundles \cite[Chapter~I]{N}.
He extends in \cite[Chapter~III]{N}  his construction to non-proper curves by using parabolic bundles. Lacking desingularization in characteristic $p>0$ makes it difficult to generalize his construction to the higher dimensional case.  If $k$ has characteristic $0$, then, as already mentioned,  $\pi^N(X,x)=\pi^{\rm \et}(X,x)$ is the Tannaka group-scheme of the category of finite flat connections \cite[Section~2]{EH}, or, equivalently, of the category of $\sO_X$-coherent $\sD_X$-modules
with finite monodromy group-scheme. 

Our construction (see Section 3, most particularly Definition \ref{defn3.2}) generalizes on a smooth variety defined over a  perfect characteristic $p>0$ field $k$  the construction of the category of  flat connections 
({\it loc. cit}) in characteristic $0$, and the construction of the stratified bundles ({\it loc. cit.}) in characteristic $p>0$. We now explain the main idea.

For $i\in \N$, let us define inductively the relative Frobenius $F^{(i)}: X^{(i)}\to X^{(i+1)}$ over $k$  in the usual manner. As $k$ is assumed to be perfect, one defines $X^{(-1)}=X\otimes_{k, F_{k}^{-1}} k$
where $F_k: \Spec k\to \Spec k$ is the absolute Frobenius of $k$, together with the relative Frobenius 
$F^{(-1)}: X^{(-1)}\to X^{(0)}$. Then one iterates to define inductively $F^{(i)}: X^{(i)}\to X^{(i+1)}$ for $i\in \Z, i<0$. 
 For 
$a,b \in \Z, a<b$ 
 we define $F^{(a,b)}: X^{(a)}\xrightarrow{F^{(a)}\circ \ldots \circ F^{(b-1)}} X^{(b)}$.

Recall that a stratified bundle is a sequence $(E^{(i)}, \sigma^{(i)}, i\in \N)$, where $E^{(i)}$ is a bundle on $X^{(i)}$, $\sigma^{(i)}: E^{(i)}\xrightarrow{\cong} F^{(i) *}E^{(i+1)}$ is a $\sO_{X^{(i)}}$-isomorphism. 
For $t\in \N, t\neq 0$,
we define an object of ${\sf Strat}(X, t )$ to be 
a sequence $(E^{(i)}, \sigma^{(i)}, i\in \N)$, where $E^{(i)}$ is a bundle on $X^{(i)}$, $\sigma^{(i)}: E^{(i)}\xrightarrow{\cong} F^{(i) *}E^{(i+1)}$ is a $\sO_{X^{(i)}}$-isomorphism for all $i\ge 1$, but for $i=0$,
$\sigma_0: F^{(-t,0)*}E^{(0)}\xrightarrow{\cong} F^{(-t,1) *}E^{(1)}$ is a $\sO_{X^{(-t)}}$-isomorphism.  The morphisms are the  ones between the bundles which respect all the structures.  We show (Theorem \ref{thm3.3}) that the obvious functor ${\sf Strat}(X, t)\subset {\sf Strat}(X, t+1)$, which assigns  $(E_i,  F^{(-t-1)*}\sigma_0, \sigma_i, i\ge 1)$ to  $(E_i, \sigma_0, \sigma_i, i\ge 1)$, 
induces a full embedding of Tannaka categories, where the fiber functor is simply the restriction of $E^{(0)}$ to the rational point $x$. Then ${\sf Strat}(X,\infty)$  is defined as the inductive limit over $t\to \infty$ of the categories ${\sf Strat}(X,t)$  (Corollary \ref{cor3.4}). In order to show that the Tannaka group-scheme  $\pi^{{\rm alg},\infty}(X,x)$ of ${\sf Strat}(X,\infty)$ surjects  onto $\pi^N(X,x)$, we use a slight modification of Nori's reconstruction theorem \cite[Chapter~I, Proposition~2.9]{N} of a torsor $h: Y \to X$ under a finite group scheme $G$ out of the induced functor $h^{\#}: {\sf Rep}_k(G)\to {\sf Coh}(X)$ which assigns to a finite dimensional $k$-linear representation $V$ of $G$ the vector bundle on $X$ which is defined by flat descent for $h$ on $\sO_Y\times_k V$  (Theorem \ref{thm2.4}). 

This allows us to define  the group-scheme homomorphism $\pi^{{\rm alg},\infty}(X,x)\to \pi^N(X,x)$ (Theorem \ref{thm4.3}). In order to show that this map induces the profinite quotient, we in particular use the categorial translation of injectivity and surjectivity of homomorphisms of Tannaka group-schemes  (\cite[Proposition~2.12]{DM}).  
  \\ \ \\
{\it Acknowledegements:} This work started while the second author was a guest of the ERC Advanced Grant 226257 in Essen.  Nguy{\^e}n Duy T\^an pointed out an inaccuracy in an earlier version of this work. The authors  had useful discussions with Alexander Beilinson and Ph\`ung H\^o Hai, who came in February 2010 to the conference on Algebraic Geometry and Arithmetic in Essen,  dedicated to the memory of Eckart Viehweg.

\section{ Nori's fundamental group-scheme}

Let $k$ be a field of characteristic $p>0$ and $X$ be a $k$-scheme. Let $x\in X(k)$ be a rational point and $i_x: x\to X$ be the closed embedding. 

Nori \cite[Chapter~II]{N} defines the category ${\sf N}(X,x)$ of triples $(Y\stackrel{f}{\to}X,G,y)$ where
\begin{enumerate}
 \item[(a)]$G/k$ is a finite group scheme,
\item[(b)]$f: Y\to X$ is a $G$-torsor,
\item[(c)]$y$ is a $k$-point of $Y$ lying above $x$.
\end{enumerate}
A morphism between two such triples $ (Y_i\stackrel{f_i}{\to}X,G_i,y_i)\ \ \ i=1,2$, 
is a pair $(\phi:G_1\to G_2,\psi:Y_1\to Y_2)$ such that $\psi$ an $X$-morphism which is $\phi$-equivariant and $\psi(y_1)=y_2$. Nori shows \cite[Chapter~II, Proposition~2]{N} that if  $X$ is reduced and geometrically connected, then the projective limit 
 $ \underset{{\sf N}(X,x)}{\ilim}G $ exits. He defines
\begin{defn} \label{defn2.1}
 Let $X$ be a reduced geometrically connected $k$-scheme, then it's Nori fundamental group-scheme is the profinite $k$-group-scheme
$$\pi^N(X,x)=\varprojlim_{{\sf N}(X,x)} G.$$
 
\end{defn}

 Since  giving a rational point $y\in f^{-1}(x)$ is the same as giving a trivialization $f^{-1}(x)\cong_k G$, ${\sf N}(X,x)$ is equivalent to the category of triples $(h: Y\to X, G, f^{-1}(x)\cong_k G)$, where the morphisms between two such objects are defined by $X$- torsor morphisms which respect the trivialization. We will not need this slightly different phrasing.  

\begin{defn} \label{defn2.2}
 Let $G$ be a finite $k$-group-scheme, and let $h: Y\to X$ be a $G$-torsor. Then it induces a functor $h^{\#}: {\sf Rep}_k(G)\to {\sf Coh}(X)$ which assigns to a finite dimensional $k$-representation $V$ the bundle on $X$ which comes by flat descent from $\sO_Y\otimes_k V$. 
\end{defn}
\begin{eig} \label{propies2.3}
\begin{itemize}
 \item[1)] The functor $h^{\#}$ defined in Definition \ref{defn2.2} is  exact, $k$-linear and compatible with the tensor structure. Thus it is a {\it fiber functor} in the sense of Deligne \cite[1.9]{D}. Since ${\sf Rep}_k(G)$ is a Tannaka category, it follows \cite[Corollaire~2.10]{D} that $h^{\#}$ is faithful.
\item[2)] The functor $i_x^*: {\sf Coh}(X)\to {\sf Vec}_k$ defined as the restriction to the rational point, with values in the category of finite dimensional $k$-vector spaces, is a fiber functor on the subcategory of vector bundles. 
The composite functor $i_x^*\circ h^{\#}: {\sf Rep}_k(G)\to {\sf Vec}_k$ is a fiber functor.  
\item[3)] Let $h_i:Y_i\to X$ be $G_i$ torsors where $i=1,2$. Let $\phi:G_1\to G_2$ be a group homomorphism and $\psi:Y_1\to Y_2$ be an equivariant map with respect to $\phi$. We denote by $\phi^*$  the induced functor ${\sf Rep}_k(G_2)\to {\sf Rep}_k(G_1)$. Then one has the equality $ h_2^{\#}= h_1^{\#}\circ \phi^*$ of functors. Indeed, if $V$ is a $G_2$-representation, 
$\psi^*: \sO_{Y_2}\otimes_k V\to \psi_*(\sO_{Y_1}\otimes_k \phi^*(V))$ induces a $\sO_X$-linear map $h_2^{\#}(V)\to h_1^{\#}(V)$ between those two vector bundles, which, after composing with $i_x^*$, is the identity on $V$. So 
$h_2^{\#}(V)= h_1^{\#}\circ \phi^* (V)$.
 \item[4)] Let $h:Y\to X$ be a $G$-torsor,  let $b:X'\to X$ be a morphism, and let $x'\in X'(k)$ be a rational point with $b(x')=x$. Let $Y'=Y\times_XX'\to X'$ and $h':Y'\to X'$ denote the projection. Then one has the equality  $ b^*\circ h^{\#} = h'^{\#}$ of functors. Indeed, denoting by $b': Y'\to Y$ the induced morphism,  if $V$ is a $G$-representation, $(b')^*: \sO_Y\otimes_k V\to (b')_* \sO_{Y'}\otimes_k V$ induces $\sO_{X'}$-linear map $b^*\circ h^{\#}(V)\to (h')^{\#}(V)$ between vector bundles, which is the identity on $V$ after composing with $i_{x'}$. So $b^*\circ h^{\#}=(h')^{\#}$.
\end{itemize}
\end{eig}

The following is a direct consequence of \cite[Proposition~2.9]{N}.
\begin{thm}\label{thm2.4}
Let $G$ be a finite $k$-group-scheme and let $F: {\sf Rep}_k(G)\to {\sf Coh}(X)$ be a fiber functor such that $i_x^*\circ F$  is  the forgetful functor $F_G: {\sf Rep}_k(G)\to {\sf Vec}_k$. 
Then there exists a unique object  $(Y\stackrel{h}{\to} X,G,y)$ of ${\sf N}(X,x)$  such that $F=h^{\#}$
 and $(h^{-1}(x), y)=(G,1).$ For  any other object $( Y'\xrightarrow{h'} X, G, y')\in {\sf N}(X,x)$ such that 
$ F=h^{\#}$,  there exists a unique isomorphism in  ${\sf N}(X,x)$ between $(Y\stackrel{h}{\to} X,G,y)$ and $( Y'\xrightarrow{h'} X, G, y')$.
\end{thm}
\begin{proof}
By Nori's reconstruction theorem \cite[Proposition~2.9]{N}, $F(k[G])$, where $k[G]$ is the regular representaton of $G$, is a finite $\sO_X$-algebra. The $G$-torsor $h:Y\to X$ is defined to be $\Spec_X F(k[G])$. By Property \ref{propies2.3} 2),  $i_x^*\circ F(k[G])=F_G(k[G])=k[G]$. 
 Said differently, $h^{-1}(x)=\Spec_x k[G]=G$.  Then $y$ is the rational point of $h^{-1}(x)$ which is $1\in G$. By the unicity in {\it loc. cit.}, $h$ is uniquely defined. If $y'=g\in h^{-1}(x)(k)$ is another rational point, then multiplication $g: Y\to Y$ by $g$, together with the conjugation $G\to G, h\mapsto ghg^{-1}$ defines an isomorphism $(h: Y\to X, G, y)\to (h: Y\to X, G, y')    $ in ${\sf N}(X,x)$.
\end{proof}

\section{The category of generalized stratified bundles}

The aim of this section is to define the category of {\it generalized stratified bundles}. We start with some notations. 

\begin{notas} \label{notas3.1}
Let $k$ be a perfect field of characteristic $p>0$, $X$ be a {\it smooth} scheme over $k$ which is geometrically irreducible.  

For $i\in \N$, we define  inductively the relative Frobenius $F^{(i)}: X^{(i)}\to X^{(i+1)}$ over $k$ in the usual manner, 
by defining $X^{(0)}=X$, $X^{(i+1)}$ to be the fiber product of $X^{(i)} \otimes_{k, F_k} k $ over the absolute Frobenius $F_k: \Spec k\to \Spec k$ of $k$, and $F^{(i)}$ to be the factorization of the absolute Frobenius $F_{X^{(i)}}: X^{(i)} \to X^{(i)}$ morphism. 

For $i\in \Z, i<0,$ we define inductively $F^{(i)}: X^{(i)}\to X^{(i+1)}$ over  $k$  as follows. First we set  $X^{(-1)}=X\otimes_{ F_{k}^{-1}} k$. Then we define  $F^{(-1)}: X^{(-1)}\to X$ to be  the relative Frobenius.  Similarly, we  define $X^{(-i-1)}=X^{(-i)}
\otimes_{ F_{k}^{-1}} k  $ together with the relative Frobenius $F^{(-i-1)}: X^{(-i-1)}\to X^{(-i)}$ over $k$.  

For 
$a,b \in \Z, a<b$ 
 we define $F^{(a,b)}: X^{(a)}\xrightarrow{F^{(a)}\circ \ldots \circ F^{(b-1)}} X^{(b)}$.
\end{notas}

Recall that a {\it stratified bundle} (see \cite[Section~1]{G}) is a sequence $(E^{(i)}, \sigma^{(i)}), i\in \N$, where $E^{(i)}$ is a $\sO_X$-coherent sheaf on $X^{(i)}$, $\sigma^{(i)}: E^{(i)}\xrightarrow{\cong} F^{(i) *}E^{(i+1)}$ is a $\sO_{X^{(i)}}$-isomorphism. One defines the {\it category} ${\sf Strat}(X)$ {\it of stratified bundles} by defining $${\rm Hom}( (D^{(i)}, \tau^{(i)}), (E^{(i)}, \sigma^{(i)}))$$
to be set of sequences $f_i: D^{(i)}\to E^{(i)}$ of morphisms of $\sO_{X^{(i)}}$-coherent sheaves, which commute with all the $\sigma_i$ and $\tau_i$. It is a fact ({\it loc. cit.}) that if $(E^{(i)}, \sigma^{(i)}, i\in \N)$ is a stratified sheaf, the $E^{(i)}$ are all locally free, and if $f=(f)_i, i\in \N$ is a morphism of stratified sheaves, then $f_i$ are vector bundle maps (i.e. locally split), so the categroy is abelian, rigid, and monoidal. Moreover, the Hom-sets are finite dimensional $k$-vector spaces. As $X$ is geometrically irreducible, the unit object ${\mathbb I}=(\sO_X, {\rm Id}), i\in \N$ fulfills ${\rm End}({\mathbb I})=k$. If now $X$ is endowed with a rational point $x\in X(k)$, then $\omega_x: {\sf Strat}(X)\to {\sf Vec}_k, ( E^{(i)}, \sigma^{(i)}) \mapsto E_0|_x$ is a fiber functor in the sense of Deligne \cite[1.9]{D}, and thus yields the structure of a Tannaka category on  ${\sf Strat}(X)$. A fundamental property due to dos Santos is that the corresponding Tannaka $k$-group-scheme
${\rm Aut}^{\otimes}(\omega_x)$ is pro-smooth  (\cite[Corollary~12]{dS}, \cite[Corollary~7]{dSOW}).

\begin{defn} \label{defn3.2} Let $t\geq 0$ be an integer. 
A $t$-{\it stratified bundle}  is a sequence $$(E^{(i)}, \sigma^{(i)}, i\in \N),$$ where $E^{(i)}$ is a $\sO_X$-coherent sheaf on $X^{(i)}$, $$\sigma^{(i)}: E^{(i)}\xrightarrow{\cong} F^{(i) *}E^{(i+1)}$$ is a $\sO_{X^{(i)}}$-isomorphism for $i\ge 1$ and for $i=0$, $$\sigma^{(0)}: F^{(-t,0)*}E^{(0)}\xrightarrow{\cong} F^{(-t,1) *}E^{(1)}$$ is a $\sO_{X^{(-t)}}$-isomorphism.

One defines the {\it category} ${\sf Strat}(X,t)$ {\it of} $t$-{\it stratified bundles} by defining $${\rm Hom}( (D^{(i)}, \sigma^{(i)}), (E^{(i)}, \tau^{(i)}))$$
to be set of sequences $f_i: D^{(i)}\to E^{(i)}$ of morphisms of $\sO_X$-coherent sheaves, which commute with all the $\sigma_i$ and $\tau_i$.

In particular, ${\sf Strat}(X,0)={\sf Strat}(X)$.
\end{defn}

\begin{ex}\label{ex1}
We now give an example of a non-trivial $1$-stratified bundle on $X=\A^1_k=\Spec(k([x])$. Thus $X^{(i)}=\Spec(k[x_i])$ where the relative Frobenius  $X^{(i)}\to X^{(i+1)}$ is induced by $x_{i+1}\to x_i^p$. For simplicity let us assume $p={\rm char}(k)=2$. Let $V$ be a $2$-dimensional vector space over $k$ with basis $e_1,e_2$. Define $$E^{(i)} = \sO_{X^{(i)}}\tensor_kV \ \ \forall \ i \ \geq 0$$ and $$\sigma^{(i)}:E^{(i)}\to F^{(i)*}E^{(i+1)}, \ \ i\geq 1$$  to be the isomorphism induced by the identity on $V$. We define 
$$\sigma^{(0)}:F^{(-1,0)*}E^{(0)} \to F^{(-1,1)*}E^{(1)}$$
to be the isomorphism defined by sending $$e_1 \to e_1, \ \ \ e_2 \to x_{-1}e_1 +e_2 .$$ 
We claim that the $-1$-stratified bundle thus defined is not isomorphic to the trivial stratified bundle of rank $2$. If indeed this were the case, then we would have a $k[x]$-module automorphism $\phi:k[x]\tensor_k V \to k[x]\tensor_kV$, such that $$\phi\tensor_{k[x]}k[x_{-1}]=\sigma^{(0)}.$$
This is impossible since $x_{-1}$ is not contained in $k[x]$. It can be shown (see \eqref{ex2}) that this $-1$-stratified bundle ``arises''  from the non-trivial $\alpha_p$ torsor on $\A^1_k$  defined by the relative Frobenius of $\A^1_k$.
\end{ex}

\begin{thm} \label{thm3.3} The notations are as in \ref{notas3.1}.
\begin{itemize}
\item[1)] For every integer $t\geq 0$, ${\sf Strat}(X,t)$ is a 
$k$-linear, abelian, rigid, tensor category. 
\item[2)] The functor 
\ga{}{(+): {\sf Strat}(X, t)\subset {\sf Strat}(X, t+1) \notag\\
   (E_i, \sigma_0, \sigma_i, i\ge 1)\mapsto (E_i,  F^{(-t-1)*}\sigma_0, \sigma_i, i\ge 1),\notag} 
induces a full faithful embedding of $k$-linear, abelian, rigid, tensor categories.
\item[3)]
If $x\in X(k)$ is a rational point, the functor 
\ga{}{\omega_x: {\sf Strat}(X,t)\to {\sf Vec}_k \notag \\
 ( E^{(i)}, \sigma^{(i)}) \mapsto E_0|_x \notag} is a fiber functor, which makes $({\sf Strat}(X, t),\omega_x)$ a Tannaka category. 

\end{itemize}
\end{thm}
\begin{proof}
We show 1). Since ${\sf Strat}(X,0)={\sf Strat}(X)$, we assume $t>0$.  If $(E^{(i)}, \sigma^{(i)}, i\in \N)$ is an object in ${\sf Strat}(X, t)$, then
 $(E_+^{(i)}=E^{(i+1)}, \sigma_{+}^{(i)}=\sigma^{(i+1)}, i\in \N)$
 is an object ${\rm Ver}(E^{(i)}, \sigma^{(i)}, i\in \N)\in {\sf Strat}(X^{(1)})$. Since $E^{(i)}$ is locally free, by the isomorphism $\sigma^{(0)}$, $F^{(-t,0)*}E^{(0)}$ is locally free. Since $X$ is smooth, the relative Frobenius is flat, thus by flat descent, $E^{(0)}$ is locally free as well. So ${\sf Strat}(X)$ is rigid and monoidal. On the other hand, 
\ml{3.1}{{\rm Hom}((D^{(i)}, \tau^{(i)}, i\in \N), (E^{(i)}, \sigma^{(i)}, i\in \N)) \\ \subset 
{\rm Hom}({\rm Ver}(D^{(i)}, \tau^{(i)}, i\in \N), {\rm Ver}(E^{(i)}, \sigma^{(i)}, i\in \N)) }
 and is obviously a $k$-vector space. So the Hom-sets are finite dimensional $k$-vector spaces. Moreover, any morphism $f=(f^{(i)}, i\in \N)$ is such that $f^{i}, i\ge 1$ is a morphism of vector bundles. Thus by the ismorphisms $\tau^{(0)}, \sigma^0$, Ker, Im and Coker of $f^{(0)}$ are pulled back to vector bundles on $X^{(-t)}$ via $F^{(-t,0)}$, thus by flat descent again, there are vector bundles on $X$. We conclude that ${\sf Strat}(X,t)$ is an abelian category. This shows 1).

2) follows immediately from the factorization of \eqref{3.1} through (+).

We show 3): the point $x\in X(k)$ maps to $x^{(1)}\in X^{(1)}(k)$, and the map $x\to x^{(1)}$ is the identity on the residue fields $k(x)=k(x^{(1)})=k$.  If $0\to A\to B\to C\to 0$ is an exact sequence in ${\sf Strat}(X,t)$, then 
$0\to {\rm Ver}(A)\to {\rm Ver}(B)\to {\rm Ver}(C)\to 0$ is an exact sequence in ${\sf Strat}(X^{(1)})$, thus $0\to \omega_{x^{(1)}}({\rm Ver}(A))\to \omega_{x^{(1)}}({\rm Ver}(B))\to \omega_{x^{(1)}}({\rm Ver}(C))\to 0$
is an exact sequence in ${\sf Vec}_k$. But 
\ga{3.2}{ \omega_{x^{(1)}}({\rm Ver}(A)) =    \omega_{x}(A).  }
This shows that $\omega_x$ is exact. Furthermore, $\omega_x$ is obviously $k$-linear and compatible with the tensor structure. This finishes the proof. 

\end{proof}

\begin{cor} \label{cor3.4}
 Let the notations be as in Theorem \ref{thm3.3}. The category
$${\sf Strat}(X, \infty)=\varinjlim_{+, t\in \N} {\sf Strat}(X,t)$$ is a $k$-linear, abelian, rigid tensor category, on which, if $X$ has a rational point $x\in X(k)$, the functor $\omega_x$ is a fiber functor. 
\end{cor}

\begin{defn} \label{defn3.5} The notations are as in Theorem \ref{thm3.3}.
\begin{itemize}
\item[1)] We define $\pi^{{\rm alg}}(X,x)$ to be the Tannaka $k$-group scheme ${\rm Aut}^{\otimes}(\omega_x)$ of $({\sf Strat}(X), \omega_x)$.
\item[2)] We define $\pi^{{\rm alg}, \infty}(X,x)$ to be the Tannaka $k$-group scheme ${\rm Aut}^{\otimes}(\omega_x)$ of $({\sf Strat}(X, \infty), \omega_x)$. 
\end{itemize}
\end{defn}
The functor $(+): {\sf Strat}(X) \to {\sf Strat}(X, \infty)$ defines 
the homomorphism 
\ga{3.3}{(+)^*: \pi^{{\rm alg}, \infty}(X,x) \to \pi^{{\rm alg}}(X,x).}
\begin{lem} \label{lem3.6}
The homomorphism $(+)^*$ in  \eqref{3.3} is faithfully flat. \end{lem}
\begin{proof}
We apply  \cite[Proposition~2.21]{DM}. As     (+) is fully faithful, the lemma is equivalent to saying that if $A$ is an object on ${\sf Strat}(X)$, and $B\subset (+)A$ is a subobject in ${\rm Strat}(X, \infty)$, then  there is a subobject $B'\subset A$ in ${\sf Strat}(X)$ such that $B=(+)B'$. One has that ${\rm Ver}(B)\subset {\rm Ver}(A)$ is a subobject in ${\sf Strat}(X^{(1)})$, thus $F^{(0) *}B^{(1)}\subset A^{(0)}$ is a subvector bundle with the property that $F^{(-t,0)*}\circ F^{(0) *}B^{(1)}=  F^{(-t,1)*}B^{(1)}=
F^{(-t,0)*}B^{(0)}$. Thus $B'=(F^{(0) *}B^{(1)}, B^{(i)}, i\ge 1, F^{(0) *}, \sigma^{(i)}, i\ge 1)\subset A$ is a subobject of $A$  such that $(+)B'=B$. This finishes the proof. 
\end{proof}

\section{Comparison of $\pi^{{\rm alg},\infty}(X,x)$ with $\pi_1^N(X,x)$}
In order to achieve the comparison, we start with a construction.  
\begin{construction}\label{const1} 
The notations are as in \ref{notas3.1}, and $x\in X(k)$ is a rational point. 
Let $(h:Y\to X, G, y)$ be an object of ${\sf N}(X,x)$. Using this object, we construct a tensor functor 
$$h^*:{\sf Rep}_k(G) \to {\sf Strat}(X,\infty)$$ 
together with a factorization of functors
\ga{4.1}{ \xymatrix{{\sf Rep}_k(G)\ar[r]^{h^*} \ar[dr]_{F_G} & {\sf Strat}(X,\infty)\ar[d]^{\omega_x}\\
 & {\sf Vec}_k
} }
Here  $F_G: {\sf Rep}_k(G) \to {\sf Vec}_k$ is the forgetful functor.

Recall that if $G$ is a finite $k$-group-scheme, there is an exact sequence of finite $k$-group schemes $1\to G_0\to G\to G_{{\rm \et}}\to 1$, where $G_0$ is the $1$-component of $G$  and $G_{{\rm \et}}$ is  \'etale. 
Furthermore, as $k$ is perfect,  $G_{{\rm red}}\subset G$ is a closed subgroup-scheme and the composite $G_{{\rm red}}\xrightarrow{\iota}  G\to G_{{\rm \et}}$ is an isomorphism. Thus $\iota$  yields on  $G$  the structure of a  semi-direct product of $G_{{\rm \et}}$ by $G_0$. 
 The construction of 
$h^*$ will be such  that the image of $h^*$ is contained in ${\sf Strat}(X,t)$, where $t$ is a natural number such that the image of the  $k$-group-scheme homomorphism $G^{(-t)}\to G$ is equal to $G_{{\rm \et}}$.

Let $V$ be a finite dimensional $k$-representation of  $G$.
We set
\ga{4.2}{ E^{(0)}=h^{\#}(V).}        
 For $i\in \N\setminus \{0\}$, the relative Frobenius is an isomorphism of the \'etale $k$-group-schemes 
\ga{4.3}{ F^{(0,i)}: G_{{\rm \et}} \xrightarrow{\cong}  G^{(i)}_{{\rm \et}}.}
Thus $ \iota(G)\circ F^{(0,i)-1}: G^{(i)}_{{\rm \et}}\subset G$ is a closed embedding and composing with it defines a $G^{(i)}_{{\rm \et}}$-action on $V$.  Since $h:Y\to X$ is a $G$-torsor, for $i\geq 0$, $h^{(i)}: Y^{(i)}\to X^{(i)}$ is also a $G^{(i)}$-torsor. Let $h^{(i)}_{\rm \et}: 
Y^{(i)}_{\rm \et}\to X^{(i)}$ be the induced $G^{(i)}_{\rm \et}$-torsor obtained by moding out by $G^{(i)}_0$. We define
\ga{4.4}{ E^{(i)}=(h^{(i)}_{\rm \et})^{\#}(V).     }
One has 
\ga{4.5}{\sigma^{(i)}: E^{(i)} \xrightarrow{=}F^{(i)*} E^{(i+1)} , \ i\in \N\setminus \{0\}.}
The object $h^*(V) \in {\sf Strat}(X,t)$ which we wish to construct will have the property
\ga{4.6}{{\rm Ver}( h^*(V))=(E^{(i)}, \sigma^{(i)}, i\ge 1).}
It remains to define $\sigma^{(0)}$.
By definition,  
\ga{4.7}{F^{(0)*}E^{(1)}=  (h^{(0)}_{\rm \et})^{\#}(V)  = (h_{\rm \et})^{\#}(V)   .}
Let
$t$ be a natural number such that the image of  $G^{(-t)}\to G$ is equal to $G_{{\rm \et}}$.
One has the following commutative diagram of  $k$-varieties.
\ga{4.8}{\xymatrix{
Y^{(-t)} \ar[rr]^{F^{(-t,0)}} \ar[dd] \ar[rd] &         &  Y \ar[dd] \ar[rd] &  \\&  Y^{(-t)}_{\rm \et} \ar[ld]_{h^{(-t)}} \ar@{-->}[ru]^{\exists ! \ \lambda} \ar[rr]^{F^{(-t,0)}\  } &  & Y_{\rm \et}\ar[ld]^{h} \\
X^{(-t)} \ar[rr]^{F^{(-t,0)}} 		  &  &  X  & 
}}
The morphism $F^{(-t,0)}: Y^{(-t)}\to Y$ is equivariant under  $F^{(-t,0)}:  G^{(-t)}\to G$. Likewise, the 
 morphism $F^{(-t,0)}: Y^{(-t)}_{\rm \et}\to Y_{\rm \et}$
is equivariant under  $F^{(-t,0)}:  G^{(-t)}_{\rm \et}\to G_{\rm \et}$.
The commutativity of the diagram implies that 
\ga{4.9}{\lambda^*(\sO_Y\tensor_k V)= F^{(-t,0)*}(\sO_{Y_{\rm \et}}\tensor_kV)=
F^{(-t,1)*}(\sO_{Y^{(1)}_{\rm \et}}\tensor_kV)
}
equivariantly for the action of $G^{(-t)}_{\rm \et}$.
Thus
\ga{4.10}{ (h_{\rm \et}^{(-t)})^{\#}(V)=   F^{(-t,0)*}E^{(0)}=F^{(-t,1)*}E^{(1)}.}
We define $\sigma^{(0)}: F^{(-t,0)*}E^{(0)}=F^{(-t,1)*}E^{(1)}$ to be the equality of \eqref{4.10}.

Thus, starting with $V\in {\sf Rep}_k(G)$, we have constructed an object $h^*(V)=(E^{(i)}, \sigma^{(i)}, i\in \N) \in {\sf Strat}(X,t)$. 
Clearly, any $\phi\in {\rm Hom}_{{\sf Rep}_k(G)}(V, W)$
induces $h^*(\phi)\in {\rm Hom}_{{\sf Strat}(X,t)}(h^*(V), h^*(W))$. 
This defines the functor 
\ga{4.11}{ h^*: {\sf Rep}_k(G)\to {\sf Strat}(X,\infty).}
by composing with (+). 
Moreover, one has
\ga{4.12}{ h^{*}(V)_x= (\sO_Y\tensor_k V)_y = V.}
This shows the commutativity of \eqref{4.1}.
\end{construction}

\begin{rmk}
In the above construction we use the fact that for a finite flat group scheme $G$ over a perfect field $k$, the epimorphism $G\to G_{\rm \et}$ admits a  section (necessarily unique). In other words $G_{\rm \et}$ can be canonically thought of as a subgroup scheme of $G$ via the identification $ G_{\rm red}=G_{\rm \et}$. When $k$ is not a perfect field,  $G_{\rm red}$ may not be a subgroup scheme, 
(for example,  $G=\Spec k[t]/(t^{p^2}-at^p), \ a\in k\setminus k^p$, see 
\cite[Chapter~III,~Exercice~(3.2)]{GM},)
and the above construction of $h^*$ does not make sense.  This is the  reason why we assume throughout $k$ to be perfect  perfect. We thank Nguy{\^e}n Duy T\^an for this important remark. 
\end{rmk}

\begin{ex}\label{ex2}Let $p={\rm char}(k)=2$ for simplicity and let $G=\alpha_2=\Spec(k[t]/t^2)$. Let $X=\A^1_k = \Spec(k[x])$. Let $P=\Spec(k[u])$, and $h:P\to X$ be the relative Frobenius defined by $x\to u^2$. Then $h$ is a $G$-torsor. Thus by  Construction \eqref{const1}, one has a functor 
$$h^*:{\sf Rep}_k(G)\to {\sf Strat}(X,-1).$$
 We compute now that $h^*(k[G])$ is nothing but the $-1$-stratified bundle defined in Example \ref{ex1}. Here $k[G]=k[v]/(v^2)$ is the regular representation of $G$. As in  Example \eqref{ex1}, let $X^{(i)}=k[x_i]$. Let $h^*(k[G])=(E^{(i)},\sigma^{(i)},i\in \N)$.  As all schemes are affine, we confuse coherent sheaves with corresponding modules.
Since $G_{\rm \et}$ is trivial, by definition of $h^*$ we see that 
$$ E^{(i)}= k[x_i]\tensor_k k[v]/(v^2) \ \ \ \forall \ i \geq 1$$
with $$\sigma^{(i)}:E^{(i)}\to F^{(i)*}E^{(i+1)} \ \ \ i\geq 1$$
induced by the identity map on $k[v]/(v^2)$. Then $E^{(0)}$ is 
 by definition the $k[x]$-module of invariants
of $k[u]\tensor_k k[v]/(v^2)$, where the action of  $G=\Spec k[t]/(t^2)$  is defined  by $$ u\to u+t, \ \ v\to v+t.$$
 Since $(u+v)^2=u^2=x$, one has  
  $E^{(0)}=k[x]\cdot 1\oplus k[x]\cdot (u+v)$.  On $P$ we have an identification  
$$h^*E^{(0)} = k[u]\tensor_k k[v]/(v^2)$$ 
defined by $\tau: 1\mapsto 1 \otimes 1, u+v\mapsto u\otimes 1+ 1\otimes v $.
The map $\sigma^{(0)}$ is nothing but the pull back of  $\tau$ via the  isomorphism 
$X^{(-1)}\to P$ defined  by $$k[u]\to k[x_{-1}], \ \ u\to x_{-1}.$$
We thus see that  $$\sigma^{(0)}:k[x_{-1}]\cdot 1\oplus   k[x_1]\cdot (u+v) \longrightarrow k[x_{-1}]\tensor k[v]/(v^2)$$
is defined by  $1 \mapsto 1\otimes, (u+v)\mapsto u\otimes 1+ 1\otimes v$. It is then an elementary exercise to see that the stratified bundle $h^*(k[G])$ is isomorphic  to the $-1$ stratified bundle defined in Example \ref{ex1}.
\end{ex}

\begin{lem} \label{lem4.2}
 The functor $h^*$ defined in \eqref{4.11}  is $k$-linear, exact,  compatible with the tensor structure and faithful. 
\end{lem}
\begin{proof}
 As already recalled  in the Properties \ref{propies2.3} 1),  faithfulness follows from the remaining properties. On the other hand, $k$-linearity, and compatibility with the tensor structures are straightforward. Exactness is proven as using Ver as in Theorem \ref{thm3.3} 3). Indeed, ${\rm Ver}\circ h^*$ with values in ${\sf Strat}(X^{(1)})$ is obviously exact, while  a sequence in ${\sf Strat}(X,\infty)$ is exact if and only if it remains exact after applying ${\rm Ver}$. 
\end{proof}
If $(h_i: Y_i\to X, G_i, y_i)$ are objects in ${\sf N}(X,x)$ for $i=1,2$ and $(\psi: Y_1\to Y_2, \phi: G_1\to G_2, y_1\to y_2)$ is a morphism in ${\sf N}(X,x)$, then  Property \ref{propies2.3} 3) implies that $h_2^*=h_1^*\circ \phi^*$. 
On the other hand, the projective system of $\phi$ in ${\sf N}(X,x)$ induces an inductive system $\varinjlim_{{\sf N}(X,x), \phi^*} {\sf Rep}_k (G)$ which is a Tannaka category, with the forgetful functor  $F_G$ as the fiber functor.
The Tannaka $k$-group-scheme ${\rm Aut}^{\otimes}(F_G)$ is simply $\varprojlim_{ {\sf N}(X,x), \phi} G$, which is precisely Nori's fundamental group-scheme $\pi^N(X,x)$. 
As in addition the construction is obviously functorial in  $h$, we conclude:

\begin{thm}\label{thm4.3}
Let the notations be as in Construction \ref{const1}. The functor $h^*$ defined in \eqref{4.11} for one object $(h: Y\to X, G, y)$ of ${\sf N}(X,x)$ induces a functor of  Tannakian categories 
\ga{}{\frak{h}^*: \big( \varinjlim_{{\sf N}(X,x), \phi^*} {\sf Rep}_k (G), F_G \big) \to 
\big({\sf Strat}(X,\infty), \omega_x\big),\notag}
and the Tannaka-dual  homomorphism of $k$-group-schemes  
\ga{}{ \frak{h}^{* \vee}: \pi^{{\rm alg},\infty}(X,x)\to \pi^N(X,x)\notag} which is functorial in $X$.
\end{thm}

The aim of the rest of the section is to show that the homomorphim 
 $\frak{h}^{* \vee}$ is faithfully flat and induces the profinite quotient homomorphism.

\begin{prop} \label{prop4.4} Let $(Y\stackrel{h}{\to} X,G,y)$ be an object of ${\sf N}(X,x)$. The following conditions are equivalent.
\begin{itemize}
\item[1)] The induced map $\pi^{{\rm alg},\infty}(X,x)\to G$  (see \eqref{4.11}) is an epimorphism.
\item[2)] The induced map $\pi^N(X,x)\to G$ is an epimorphism. 
\item[3)]  The functor $h^*$ in \eqref{4.11} is fully faithful and its image is closed under taking subquotients in ${\sf Strat}(X, \infty)$.
\end{itemize}
\end{prop}
\begin{proof}
The equivalence $(1) \Leftrightarrow (3)$ follows from \cite[Proposition~2.21]{DM}. Moreover, since by construction, the map $\pi^{{\rm alg},\infty}(X,x)\to G$ factors through $\pi^N(X,x)$, $(1)\Rightarrow (2)$ is obvious. 

We show $(2)\Rightarrow (3)$. 
Let $\sC$ denote the full  subcategory of ${\sf Strat}(X,\infty)$ generated by subquotients in ${\sf Strat}(X,\infty)$ of objects which are in the image of $h^*:{\sf Rep}_k(G)\to {\sf Strat}(X,\infty)$. The property 3) is equivalent to saying that $h^*: {\sf Rep}_k(G)\to \sC$ is an equivalenece of categories. By standard Tannaka formalism, $\sC$  itself is a $k$-linear, abelian, rigid tensor subcategory  of 
${\sf Strat}(X,\infty)$, thus $(\sC, \rho_x)$ is a Tannaka subcategory of 
$({\sf Strat}(X,\infty), \omega_x)$, where $\rho_x=\omega_x|_{\sC}$.

We show now  that $h^*: {\sf Rep}_k (G)\to \sC$ is an equivalence of categories. 
 Let $H={\rm Aut}(\rho_x)$ be the Tannaka $k$-group-scheme of $(\sC,\rho_x)$. We claim that the induced homomorphism $H\to G$ is a closed immersion.
This is equivalent (\cite[Proposition~2.21]{DM}) to saying that every object of $\sC$ is a subquotient in 
$\sC$ of an object  in $h^*({\sf Rep}_k(G))$, which is true since by definition of $\sC$, a subquotient in $\sC$ of objects in $h^*({\sf Rep}_k(G))$ is the same as a subquotient in ${\sf Strat}(X,\infty)$ of objects in $h^*({\sf Rep}_k(G))$. 
We conclude in particular that $H$ is a finite group scheme.   

The fiber functor (in the sense of Deligne \cite[1.9]{D}, see Properties \ref{propies2.3} 1))  $\omega_X: {\sf Strat}(X,\infty)\to {\sf Coh}(X)$ defined by 
$ (E_i,\sigma_i, i\in \N) \mapsto E_0$
restricts to the fiber functor $\rho_X: \sC \to {\sf Coh}(X)$. One has a 
commutative diagram of functors
\ga{4.13}{ \xymatrix{{\sf Rep}_k(G)\ar[r]^{h^*} \ar[dr]_{h^{\#}} & 
\sC \ar[d]^{\rho_X}\\
 & {\sf Coh}(X)
} }
and, upon applying $i_x$, \eqref{4.1} implies that $i_x\circ h^{\#}=F_G$. 
By applying Theorem \ref{thm2.4}, we obtain a morphism 
\ga{4.14}{(h_H:Y_H\to X,H,y_H) \to (h:Y\to X,G,y)}
in ${\sf N}(X,x)$.
This in turn induces a factorization of $\pi^N(X,x)\to G$ as 
\ga{4.15}{\xymatrix{
\pi^N(X,x)\ar[r]\ar[d] & G\\
H\ar[ru] & 
}}
But $\pi^N(X,x)\to G$ is assumed to be an epimorphism. Thus $H\to G$ must be an epimorphism. Since it is also a closed immersion, we conclude
\ga{4.16}{H\xrightarrow{\cong} G.}
In other words 
\ga{4.17}{h^*: {\sf Rep}_k (G)\xrightarrow{\cong} \sC.}
This finishes the proof.
\end{proof}
 Recall that $k$ is perfect.
\begin{lem}\label{et} Let $G$ be a finite $k$-group-scheme, let
 $h:Y\to X$ be a $G$-torsor.
 Then the following conditions are equivalent
\begin{itemize}
 \item[(i)] $h$ admits a reduction (necessarily unique) of structure group to $G_{\rm red}=G_{\rm \et}\subset G$. 
 \item[(ii)] For every natural number $n$, there is a $G$-torsor $h_n: Y_n\to X^{(n)}$ which pulls back via  $X\xrightarrow{F^{(0,n)}} X^{(n)}$ to $h$. 
\end{itemize}
\end{lem}
\begin{proof}
We show (i) $\Rightarrow$ (ii).  Let $h_{\rm \et}: Y_{\rm \et}\to X$ be a $G_{\rm \et}$-torsor which is a  reduction of structure of $h$ for the closed embedding $G_{\rm \et}\subset G$. Thus $Y=Y_{\rm \et}\times_{G_{\rm \et}} G$. The isomomorphism \eqref{4.3} induces a cartesian diagram
\ga{4.18}{\xymatrix{Y_{\rm \et}\ar[d]_{h_{\rm \et}} \ar[r]^{F^{(0,n)}} \ar@{}[rd]|\Box& (Y_{\rm \et})^{(n)} \ar[d]^{(h_{\rm \et})^{(n)}}\\
X\ar[r]_{F^{(0,n)}} & X^{(n)}
}
}
We set $Y_n=  (Y_{\rm \et})^{(n)}\times_{ {G_{\rm \et}}} G, \  h_n=(h_{\rm \et})^{(n)}\times_{G_{\rm \et}} G$. 

We show (ii) $\Rightarrow$ (i). For a large enough positive integer $n$, we consider the  commutative diagram similar to \eqref{4.8}:
\ga{4.19}{
\xymatrix{
Y_n^{(-n)}\ar[dd]\ar[rr]^{\gamma}\ar[rd] &  & Y_n \ar[dd]_{h_n} \\
	& \left(Y_n^{(-n)}\right)_{\rm \et} \ar@{--}[ru]^{\exists \ ! }\ar[ld] & \\
X \ar[rr] &    &  X^{(n)}
}}
We explain the terms in the diagram: with Notations \ref{notas3.1}, one has 
 $ Y_n^{(-n)} =Y_n\otimes_{F_k^{-n}}k$, thus $h_n$ induces $h_n\otimes_{F_k^{-n}}k: 
Y_n^{(-n)} \to \big( X^{(n)}\big)^{(-n)}=X$, 
which is a principal $G^{(-n)}$ bundle.
The top horizontal map $\gamma$  is equivariant with respect to  $G^{(-n)}\xrightarrow{F^{(-n,0)}} G$. Since $n$ is large, the image of $G^{(-n)}\to G$ is precisely $G_{\rm \et} \subset G$. Therefore, $\gamma$  factors uniquely through 
$\left(Y_n^{(-n)}\right)_{\rm \et}$. 
Via the identification $G_{\rm \et}^{(-n)}\xrightarrow{F^{(-n,0)}} G_{\rm \et} $, the morphism $\left(Y_n^{(-n)}\right)_{\rm \et} \to X^{(n)}$ is a $G_{\rm \et}$-torsor. The above commutative diagram shows the existence of an equivariant map  $\left(Y_n^{(-n)}\right)_{\rm \et} \to Y_n\times_{X^{(n)}}X$.  We conclude that the $G$-torsor  $ Y_n \times_{X^{(n)}}X \to X$ has a reduction of structure group to $G_{\rm \et}$. 
 
\end{proof}

\begin{thm} \label{thm4.6} Let the notations are as in \ref{notas3.1} and let $x\in X(k)$ be a rational point. Then the homomorphism
$\frak{h}^{*\vee}: \pi^{{\rm alg},\infty}(X,x)\to \pi^N(X,x)$ is the profinite quotient homomomorphism. 
\end{thm}
\begin{proof}
We have already shown in Proposition \ref{prop4.4} that the homomorphism $\frak{h}^{*\vee}$  is surjective. In order to show that $  \frak{h}^{*\vee} $ is the profinite completion homomorphism, we need to show that any epimorphism $$\phi:\pi^{{\rm alg}, \infty}(X,x)\to G,$$ where $G$ is a $k$-finite group-scheme, factors through $\pi^N(X,x)$. This is equivalent to showing that given any finite Tannaka subcategory $\sT\subset {\sf Strat}(X,\infty)$, i.e. with $G={\rm Aut}^{\otimes}(\sT, \rho_x)$ finite, where $\rho_x=\omega_x|_{\sT}$,  there exists an object $(h:Y\to X,G,y)$ in ${\sf N}(X,x)$ such that $\sT$ is the image of the functor $h^*$ constructed in \eqref{4.11}.
 We do this in two steps. \\[.1cm]
\noindent \underline{Step(1)}: For each $n\geq 0$, we consider  the fiber functor  
\ga{4.20}{\omega_{X^{(n)}}: {\sf Strat}(X,\infty)\to {\sf Coh}(X^{(n)}), \ \  (E^{(i)}, \sigma^{(i)}, i\in \N)\mapsto E^{(n)} .      } 
 It restricts to a  fiber functor
 \ga{}{P_n:\sT\to {\sf Coh}(X^{(n)}). \notag} 
Let $\delta:{\sf Rep}_k(G) \to \sT$ be the equivalence given by Tannaka categories defined by the inverse functor to the equivalence induced by $\rho_x$. Consider  $$P_n\circ \delta :{\sf Rep}_k(G)\to {\sf Coh}(X^{(n)}).$$ By Theorem \ref{thm2.4}, we obtain $G$-torsors $(h_n:Y_n\to X^{(n)})$ for every $n$, such that 
\ga{4.21}{h_n^{\#}=P_n\circ \delta  .}
Since the $G$-torsors thus obtained are unique upto isomorphism, the  equality
\ga{}{ P_n = F^{(n)*}\circ P_{n+1},\ \forall \ n \geq 1  \notag} 
implies  that the torsor $h_{n+1}$ pulls back to $h_n$. Thus by Lemma \ref{et}, each $Y_n$ admits a reduction of structure group to $G_{\rm \et}\subset G$ for all $n\geq 1$.\\[.1cm]
\noindent \underline{Step(2)}:  Composing $\delta$ with the inclusion $\sT \inj {\sf Strat}(X,\infty)$ we  obtain  a functor from ${\sf Rep}_k(G) \to {\sf Strat}(X,\infty)$. We also have the  functor $h_0^*:{\sf Rep}_k(G)\to {\sf Strat}(X,\infty)$  (see \eqref{4.11}) defined by the $G$-torsor $h_0:Y_0\to X$. In order to finish the proof we have to show that these two functors coincide. This is equivalent to saying that the following diagram of functors commutes.
\ga{4.22}{\xymatrix{
{\sf Rep}_k(G) \ar[r]^{\delta}\ar[rd]_{h_0^*} & \sT \ar[d]^{{\rm incl.}} \\
& {\sf Strat}(X,\infty)
}.}
Let $V$ be an object of ${\sf Rep}_k(G)$. We will show that there is an isomorphism between  $i(V)$ and $h_0^*(V)$, which is functorial in $V$. This will finish the proof. Let $\delta(V)=(\delta(V)^{(n)},\sigma^{(n)},n\in \N)$ and $h_0^*(V)=(E^{(n)},\tau^{(n)},n\in \N)$. 

We let  $h_{n,{\rm \et}}:Y_{n,{\rm \et}}\to X^{(n)}$ be the $G_{{\rm \et}}$-torsor induced by $h_n$ for $n\ge 1$. Note that by construction  \ref{const1} of the functor $h_0^*$,  one has 
\ga{4.23}{E^{(n)}=h_{n,{\rm \et}}^{\#}(V)\ \forall \ n\geq 1 \ \ \ \text{and} \ \  E^{(0)}=h_0^{\#}(V).} 
On the other hand, by definition of the functors $P_n$, $$P_n(i(V))=i(V)^{(n)}$$
Thus by \eqref{4.21}, one has
 \ga{4.24}{i(V)^{(n)}=h_n^{\#}(V) \ \forall \ n\geq 0.} 
But as explained before, for every $n\geq 1$, $h_n:Y_n\to X^{(n)}$ admits a reduction of structure group to  $G_{{\rm \et}}$. Thus by Proposition \ref{propies2.3}(3),
\ga{4.25}{h_n^{\#}(V)= h_{n,{\rm \et}}^{\#}(V) \ \ \forall\  n\geq 1.}
Thus we conclude 
\ga{4.26}{i(V)=h_0^*(V).}

\end{proof}
If $\sT$ is any $k$-linear, abelian, rigid tensor category, together with a neutral fiber functor $\omega: \sT \to {\sf Vec}_k$, we denote by $\sT^{{\rm fin}}$ the full subcategory spanned by objects $E$ which have the property that the full tensor subcategory $\langle E\rangle\subset \sT$ spanned by $E$ and its dual $E^\vee$ has a finite  Tannaka group scheme ${\rm Aut}^{\otimes}(
\langle E \rangle, \omega|_{\langle E \rangle})$. So by construction, Theorem \ref{thm4.6} has the following consequence:

\begin{cor}\label{cor4.7}
 With the notations as in Theorem \ref{thm4.6}, the full embedding $$
{\sf Strat}(X,x)^{{\rm fin}}\subset {\sf Strat}(X,x)$$ induces via the fiber functor $\omega_x$ the quotient homomorphism $$\pi^{{\rm alg}, \infty}(X,)\to \pi^N(X,x).$$ 
\end{cor}

\end{document}